\documentclass{amsart}
\usepackage{hyperref}
\usepackage{amsmath,amsfonts,amsthm,amssymb}

\theoremstyle{plain}
\newtheorem{theorem}{Theorem}
\newtheorem{hypothesis}{Hypothesis}
\newtheorem{lemma}{Lemma}
\newtheorem*{clemma}{Cut-Off Lemma}
\newtheorem{proposition}{Proposition}
\newtheorem{corollary}{Corollary}

\theoremstyle{remark}
\newtheorem*{remark}{Remark}
\newtheorem*{remarks}{Remarks}

\theoremstyle{definition}

\DeclareMathOperator{\const}{constant}

\DeclareMathOperator{\e}{e}

\newcommand{\R}{\mathbb{R}}
\newcommand{\dd}{\,\mathrm{d}}

\begin{document}
\title[A maximum principle for systems with variational structure]{A maximum principle for systems with variational structure and an application to standing waves}
\author{Nicholas D.\ Alikakos}
\address{Department of Mathematics\\ University of Athens\\ Panepistemiopolis\\ 15784 Athens\\ Greece \and Institute for Applied and Computational Mathematics\\ Foundation of Research and Technology -- Hellas\\ 71110 Heraklion\\ Crete\\ Greece}
\email{\href{mailto:nalikako@math.uoa.gr}{\texttt{nalikako@math.uoa.gr}}}
\thanks{The first author was partially supported through the project PDEGE -- Partial Differential Equations Motivated by Geometric Evolution, co-financed by the European Union -- European Social Fund (ESF) and national resources, in the framework of the program Aristeia of the `Operational Program Education and Lifelong Learning' of the National Strategic Reference Framework (NSRF)}
\author{Giorgio Fusco}
\address{Dipartimento di Matematica Pura ed Applicata\\ Universit\`a degli Studi dell'Aquila\\ Via Vetoio\\ 67010 Coppito\\ L'Aquila\\ Italy}
\email{\href{mailto:fusco@univaq.it}{\texttt{fusco@univaq.it}}}
\date{}

\begin{abstract}
We establish via variational methods the existence of a standing wave together with an estimate on the convergence to its asymptotic states for a bistable system of partial differential equations on a periodic domain. The main tool is a replacement lemma which has as a corollary a maximum principle for local minimizers.
\end{abstract}

\maketitle

\section{Introduction}\label{intro}
We consider the elliptic system
\begin{equation}\label{system}
\Delta u=W_u(u), \text{ for } u:\Omega \subset \R^n \to \R^m,
\end{equation}
where $W: \R^m \to \R$ is a $C^2$ potential and $W_u (u) := (\partial W / \partial u_1, \dots, \partial W / \partial u_m)^{\top}$. Systems of type \eqref{system} have been studied in particular in \cite{abg,bronsard-gui-schatzman,gui-schatzman,af1,a1,f,gui}, generally under symmetry hypotheses on the potential.

We assume that $W$ possesses two global minima that satisfy the following hypotheses.
\begin{hypothesis}\label{h1}
There exist $a_-\neq a_+\in\R^m$ such that
\[ 0=W(a_-)=W(a_+)<W(u),\text{ for all } u \in \R^m \setminus \{a_-,a_+\}. \]
\end{hypothesis}

\begin{hypothesis}\label{h2}
There is an $r_0>0$ such that, for $\nu \in \mathbb{S}^{m-1}$, where $\mathbb{S}^{m-1}\subset\R^m$ is the unit sphere, the map $(0,r_0] \ni r \to W(a+r \nu)$, for $a \in\{a_-,a_+\}$ has a strictly positive first derivative.
\end{hypothesis}

We are interested in globally bounded solutions of \eqref{system} and so growth conditions on $W$ at infinity are not relevant. Therefore, we have the following hypothesis.
\begin{hypothesis}\label{h3}
There exists $M>0$ such that
\[ W(s u) \geq W(u), \text{ for } s \geq 1 \text{ and }\vert u\vert=M. \]
\end{hypothesis}

We further assume that $\Omega \subset \R^n$ is a $C^2$ \emph{periodic} domain with bounded cross section. We let $x=(s,y) \in \R \times \R^{n-1}$ be the typical element of $\R^n$.
\begin{hypothesis}\label{h4}
There exist $L>0$ and $R>0$ such that
\[ (s,y) \in \Omega \text{ implies } (x\pm L,y)\in\Omega \text{ and } \vert y\vert\leq R,\]
\end{hypothesis}
For fixed $s \in \R$ we denote by $\Omega^s := \Omega \cap \left( \{s\} \times \R^{n-1} \right)$ the \emph{cross section} of $\Omega$ with the plane $s = \const$.

We also need the following technical hypothesis.
\begin{hypothesis}\label{h5}
The set $\Omega^0$ is connected.
\end{hypothesis}
This allows domains with a complicated topology with holes and other pathologies. We remark that there exist domains for which $\Omega^s = \Omega^0$ is the only connected cross section for $s\in(-L,L)$. Hypothesis \ref{h5} can be relaxed to
\begin{equation}\label{relaxed-h-5}
\Omega \cap \{ (s,y) \mid s=\sigma(y), \text{ for } \vert y \vert \leq R\} \text{ is a connected set,}
\end{equation}
where $\sigma : \{\vert y \vert \leq R \} \to \R$ is a smooth map.

For the boundary value problem
\begin{equation}\label{system-boundary}
\begin{cases}
\Delta u = W_u(u), &\text{in } \Omega,\medskip\\
\dfrac{\partial u}{\partial n} = 0, &\text{on } \partial\Omega,\medskip\\
\lim_{s\to\pm\infty,\,(s,y)\in\Omega} u(s,y) = a_\pm,
\end{cases}
\end{equation}
where $n$ is the outward normal, we establish the following result.

\begin{theorem}\label{teo-1}
Assume that $W$ and $\Omega$ satisfy Hypotheses \ref{h1}--\ref{h5}. Then there exists a classical solution $u: \Omega \to \R^m$ to the boundary value problem \eqref{system-boundary}. If $a_+$ is nondegenerate in the sense that the quadratic form $\langle D^2(W(a_+))z , z \rangle$ is positive definite, then we have exponential decay to $a_+$, that is, there exist $k, K>0$ such that
\[ \vert u(s,y) - a_+ \vert \leq K_0 \e^{-k_0 s},\text{ for } s>0. \]
A similar statement applies to $a_-$.
\end{theorem}

We note here that for the equation $\Delta u = W_u (x,u)$, for potentials $W((s,y),u)$ periodic in $s$, a theorem analogous to Theorem \ref{teo-1} can be also established under a natural extension of the above hypotheses that take into account the $s$-dependence of the potential.

A basic fact about \eqref{system} under Hypothesis \ref{h4}, that has to be dealt with but which is also exploited, is the $L$-translation invariance of the energy
\[ J_{\Omega} (u) = \int_\Omega \left( \frac{1}{2} \vert \nabla u \vert^2 + W(u) \right) \!\dd x \]
in the $s$ direction. The class of domains $\Omega$ that satisfy Hypotheses \ref{h4} and \ref{h5} includes the case where $\Omega$ is a flat cylinder, which is naturally associated to the ODE version of \eqref{system-boundary}, that is,
\begin{equation}\label{heteroclinic-ode}
\begin{cases}
u^{\prime\prime}=W_u(u), &\text{for } s\in\R,\medskip\\
\lim_{s\to\pm\infty}u(s)=a_\pm.
\end{cases}
\end{equation}
Solutions to \eqref{heteroclinic-ode} are also known as \emph{heteroclinic connections} (see \cite{ste} and \cite{af1}).

The present work bears a relationship to the ODE system \eqref{heteroclinic-ode}, similar to the relationship that the work \cite{bn} bears to the traveling-wave problem for scalar parabolic equations. The difference is in the way higher dimensionality is introduced. In our case we assume periodicity of the domain but we keep the equation as before. In \cite{bn} the domain is a flat cylinder but the equation is modified by including spatial convection in the $s$ direction. In the scalar case $m = 1$, existence for the boundary value problem \eqref{system-boundary} was established in \cite{bates-ren} and \cite{ren} for second and higher-order operators.

Our proof is variational and modeled after \cite{af2}. It proceeds by introducing an artificial constraint that restores compactness by eliminating the translation allowed by the periodicity of $\Omega$ and forces the appropriate behavior at infinity. The major effort is directed toward removing the constraint in the sense of showing that it is not realized. The technique for doing so cannot invoke the usual maximum principle, which does not hold in the case at hand. Therefore, our technique is purely variational. The main tool here is the Cut-Off Lemma, which is of independent interest and has as a corollary the following maximum principle.\footnote{We would like to thank Ha\"im Brezis for pointing out to us this formulation of the Cut-Off Lemma in Section \ref{section2} and for his interest in this work.} We note that connectedness is crucial here.

\begin{theorem}
Let $W: \R^m \to \R$ be $C^1$ and nonnegative. Assume that $W(a)=0$ for some $a \in \R^m$ and that there is $r_0 > 0$ such that for $\nu \in \mathbb{S}^{m-1}$ the map
\[ (0,r_0] \ni r \mapsto W(a+r\nu) \]
has a strictly positive derivative. Let $A \subset \R^n$ be an open, connected, bounded set, with $\partial A$ minimally smooth, and suppose that $\tilde{u} \in W^{1,2}(A;\R^m)\cap L^\infty(A;\R^m)$ minimizes
\[ J_A (u) = \int_A \left( \frac{1}{2} \vert \nabla u \vert^2 + W(u) \right) \!\dd x \]
subject to its Dirichlet values $u=\tilde{u}$ on $\partial A$.

Then, if there holds
\[ |\tilde{u} (x)  - a| \leq r, \text{ for } x \in \partial A, \]
for some $r > 0$ with $2r \leq r_0$, then there also holds
\[ |\tilde{u} (x)  - a| \leq r, \text{ for } x \in A. \]
\end{theorem}

The Cut-Off Lemma is a replacement result modeled after \cite{af2} and is presented in Section \ref{section2}. In Section \ref{section3} we introduce the constrained variational problem. In Section \ref{section4} the constraint is removed in two stages. First, the constraint is removed at infinity by utilizing linear estimates and then, by invoking the Cut-Off Lemma, we finish the proof.

\subsection*{Acknowledgement} The first author would like to acknowledge the warm hospitality of the Department of Mathematics of Stanford University in the spring semester of 2012, during which, part of this paper was written. Special thanks are due to Rafe Mazzeo, George Papa\-nicolaou, Lenya Ryzhik, and Rick Schoen.

\section{The Cut-Off Lemma}\label{section2}

\subsection{The polar form}
Let $A$ be an open and bounded subset of $\R^n$. For $u \in W^{1,2}(A ; \R^m) \cap L^{\infty}(A ; \R^m)$ and $a \in \R^m$, if $\rho(x):=\vert u(x)-a\vert\neq 0$, we consider the \emph{polar representation}
\begin{equation}\label{polar}
u(x) = a + |u(x)-a| \frac{u(x)-a}{|u(x)-a|} =: a+\rho (x) \nu(x),
\end{equation}
where $\vert \cdot \vert$ is the Euclidean norm and $\nu(x):=u(x)-a/|u(x)-a|$. We call $\rho$ the radial part and $\nu$ the angular part.

The purpose of this subsection is to establish rigorously the appropriate version of the identity
\begin{equation}\label{integral polar}
\int_{A} |\nabla u|^2 \dd x = \int_{A} |\nabla \rho|^2  \dd x + \int_{A} \rho^2 (x) |\nabla \nu|^2  \dd x,
\end{equation}
for $u$ as above, and also show that modifying the radial part in \eqref{polar} by setting
\[ \tilde u(x)=a+f( \rho (x) ) \nu(x), \text{ with } f(0)=0, \]
produces a $\tilde{u} \in W^{1,2}(A; \R^m) \cap L^{\infty}(A ; \R^m)$ for locally Lipschitz $f: \R \to \R$, with a corresponding formula \eqref{integral polar} (cf.\ \cite{alikakos-fusco-smyrnelis}). For our purposes in this paper it suffices to take $f(s)/s$ locally Lipschitz. We present this case because of its simplicity. Without loss of generality we take $a=0$ and set
\[ \rho(x)=|u(x)| \text{ and } \nu(x)=\frac{u(x)}{\rho(x)},  \]
on
\[ A_+:= \{ x \in A \mid \rho >0 \} \text{ and } A_0:=\{ x \in A \mid \rho =0\}.\]

\begin{proposition}
Let $w_j : A \to \R^m$ be defined by
\[ w_j:=
\begin{cases}
0, &\text{ on } A_0, \medskip\\
u_{,j} - \rho_{,j} \nu, &\text{ on } A_+,
\end{cases}
\]
where $u_{,j}:=\partial u / \partial x_j$, $\rho_{,j}:=\partial \rho / \partial x_j$. Then,
\begin{enumerate}
\item $w_j \in L^{2}(A ; \R^m)$,\medskip
\item $\langle w_j , \nu \rangle=0$ on $A_+$, where $\langle \cdot, \cdot \rangle$ is the inner product in $\R^m$,\medskip
\item there exists a measurable map $\nu^{,j}: A_+ \to \R^m$ such that $w_j=\rho \nu^{,j}$ on $A_+$.
\end{enumerate}
\end{proposition}

Note that $\rho $ is in $W^{1,2}(A)$ (cf.\ for example \cite[p.~130]{evans-gariepy}) and that $\nu^{,j}$ plays the role of $\nu_{,j}$.

\begin{proof}
Given $\varepsilon >0$, the map $\frac{1}{\rho +\varepsilon}$, a composition of  $\rho \in W^{1,2}(A) \cap L^{\infty}(A )$ and a Lipschitz map, belongs to $W^{1,2}(A) \cap L^{\infty}(A)$. From this, it follows that the map $\nu^{\varepsilon}:=\frac{u}{\rho+\varepsilon}$ is in $W^{1,2}(A ; \R^m)$, and moreover,
\begin{equation}\label{n epsilon}
\nu^{\varepsilon}_{,j}=\frac{u_{,j}}{\rho+\varepsilon}-\frac{u}{(\rho+\varepsilon)^2}\rho_{,j}.
\end{equation}
Set $w^{\varepsilon}_j=\rho \nu^{\varepsilon}_{,j}$. After multiplication by $\rho$, equation \eqref{n epsilon} becomes
\begin{equation}\label{n epsilon2}
w^{\varepsilon}_{j}=\frac{\rho}{\rho+\varepsilon}u_{,j}-\frac{\rho}{(\rho+\varepsilon)^2}u \rho_{,j} .
\end{equation}
Therefore,
\[ \lim_{\varepsilon \to 0} w^{\varepsilon}_{j}=\rho \lim_{\varepsilon \to 0}  \nu^{\varepsilon}_{,j}=
\begin{cases}
0, &\text{ on } A_0, \medskip\\
u_{,j} - \rho_{,j} \nu, &\text{ on } A_+.
\end{cases}
\]
On the other hand from equation \eqref{n epsilon2} we also have
\[
|w^{\varepsilon}_{j}| \leq
\begin{cases} 0, &\text{ on } A_0, \medskip\\
|u_{,j}|+|\rho_{,j}|, &\text{ on } A_+. \end{cases}
\]
Thus, by Lebesgue's dominated convergence theorem, there holds
\[ \lim_{\varepsilon \to 0} w^{\varepsilon}_{j}=w_j \in L^2(A; \R^m). \]
This proves (i) and, since $w^{\varepsilon}_j=\rho \nu^{\varepsilon}_{,j}$, we have $\rho \lim_{\varepsilon \to 0} \nu^{\varepsilon}_{,j}=u_{,j}-\rho_{,j}\nu$ on $A_+$, and thus $\nu^{,j}$ is defined via $\lim_{\varepsilon \to 0} \nu^{\varepsilon}_{,j}=: \nu^{,j}$ and satisfies (iii). Finally, to show (ii) we observe that from \eqref{n epsilon2} it follows that
\begin{equation}\label{n epsilon3}
\langle w^{\varepsilon}_{j},u \rangle=\frac{\rho}{\rho+\varepsilon} \langle u_{,j}, u \rangle -\frac{\rho^2}{(\rho+\varepsilon)^2} \rho \rho_{,j}=\frac{\rho}{\rho+\varepsilon} \left( 1- \frac{\rho}{\rho+\varepsilon} \right) \langle u_{,j},u \rangle .
\end{equation}
Hence, passing to the limit in \eqref{n epsilon3} gives
\[ 0=\langle w_{j},u \rangle= \langle w_{j},\rho \nu \rangle \text{ on } A_+, \]
and (ii) follows.
\end{proof}

\begin{corollary}\label{cor2}
The following identity holds.
\begin{eqnarray}\label{six} \int_{A} |\nabla u|^2 \dd x = \int_{A} |\nabla \rho|^2  \dd x + \int_{A_+} \rho^2  \sum _j \langle \nu^{,j}, \nu^{,j} \rangle \dd x.
\end{eqnarray}
\end{corollary}

\begin{proof}
Since $u_{,j}=0$ a.e.\ on $A_0$, we have
\begin{align*}
\int_{A}|\nabla u|^2 \dd x &=
\int_{A_+} |\nabla u|^2 \dd x = \int_{A_+}  \langle w_j+\rho_{,j} \nu , w_j+\rho_{,j} \nu \rangle \dd x \\
&= \int_{A_+} \bigg( \sum_j |w_j|^2+|\nabla \rho|^2 \bigg)  \dd x = \int_{A_+} \bigg( |\nabla \rho |^2+ \sum_j \rho^2 \langle \nu^{,j} , \nu^{,j}  \rangle \bigg) \dd x.
\end{align*}
\end{proof}
We remark explicitly that equation (\ref{six}) gives a precise rigorous meaning to the representation formula (\ref{integral polar}).
\begin{corollary}\label{cor3}
Let $u$ be as in \eqref{polar} above and $r \geq 0$. Let also
\[
\tilde u (x)=
\begin{cases}
a+ \min \{ \rho (x), r \} \nu(x), &\text{ for } x \in A_+, \medskip \\
a, &\text{ for } x \in A_0.
\end{cases}
\]
Then $\tilde u \in W^{1,2}(A ; \R^m) \cap L^{\infty}(A ; \R^m)$, and we have the following explicit representation of the energy
\begin{equation}\label{integral rho tilde cor}
\int_{A} |\nabla \tilde  u|^2 \dd x = \int_{A} |\nabla \tilde \rho|^2  \dd x + \int_{A_+} \tilde \rho^2  \sum _j \langle \nu^{,j}, \nu^{,j} \rangle \dd x,
\end{equation}
where $\tilde \rho (x) = |\tilde u(x)-a| = \min \{ \rho (x), r \}$ on $A$.
\end{corollary}

\begin{proof}
On $A_+$ we have
\[ \tilde u(x)=a+\frac{\min \{ \rho (x), r \}}{\rho(x)}(u(x)-a). \]
Thus, if we define
\begin{equation}\label{def g}
g(s) :=
\begin{cases}
1, &\text{ for } s \leq 0, \medskip \\
\dfrac{s+r-|s-r|}{2s}, &\text{ for } s>0,
\end{cases}
\end{equation}
then, since $\min \{ a,b \}=\frac{1}{2}(a+b-|a-b|)$, for $a,b \in \R$, we have
\[
\tilde u(x)=a +g(\rho(x))(u(x)-a), \text{ for } x \in A.
\]
Since $g$ is Lipschitz (and $\rho$ bounded) it follows that $g(\rho(\cdot))$ is in $W^{1,2}(A) \cap L^{\infty}(A)$, and therefore $\tilde u  \in W^{1,2}(A ; \R^m) \cap L^{\infty}(A ; \R^m)$, and \eqref{integral rho tilde cor} follows by Corollary \ref{cor2}.
\end{proof}
\begin{remark}
We will also need the cut-off function
\[
 \alpha (\tau):=
\begin{cases}
1, &\text{ for } \tau \leq r \ \ (r>0), \medskip \\
\dfrac{2r -\tau}{r},  &\text{ for } r \leq \tau \leq 2r, \medskip\\
0, &\text{ for } \tau \geq 2r.
\end{cases}  \nonumber
\]
Let
\[
 \tilde u (x)=
\begin{cases}
a+ \min \{ \rho (x), r \} \alpha(\rho(x)) \nu(x), &\text{ for } x \in A_+ \cap \{ \rho < 2r \},\medskip \\
a, &\text{ for } x \in A_0 \cup \{ \rho \geq 2r \}.
\end{cases}  \nonumber
\]
Set $\tilde \rho (x):=|\tilde u(x)-a|$ and $\tilde A_0=A_0 \cup  \{ \rho \geq 2r \}$, $\tilde A_+= A_+ \cap \{ \rho < 2r \}$. Then,
\begin{equation}\label{tilde u sobolev}
\tilde u  \in W^{1,2}(A ; \R^m) \cap L^{\infty}(A ; \R^m),
\end{equation}
and the analogous to \eqref{integral rho tilde cor} holds. Indeed,
\[ \tilde \rho (x)=\min \{ \rho (x), r \} \alpha(\rho(x)) \text{ on } \tilde A_+, \]
and
\[ \tilde u(x)=a +g(\rho(x)) \alpha(\rho(x))(u(x)-a), \]
with $g$ as in \eqref{def g}. Since $\alpha$ is Lipschitz, the same argument applies and renders \eqref{tilde u sobolev}.
\end{remark}

\subsection{The lemma}
We are now ready to state the main technical tool of the paper.

\begin{clemma}\label{cut-off}
Let $W$ be a $C^2$ potential, with $W: \R^m \to \R$. Assume that the map
\begin{equation}\label{hypoth}
r \mapsto W(a+r \nu) \textit{ is strictly increasing,}\tag{H}
\end{equation}
for $r \in (0, r_0]$, $r_0 >0$, with $W(a)=0$ and $W \geq 0$ otherwise.
Set
\[
J_{\Omega}(u) := \int_{\Omega} \left( \frac{1}{2} |\nabla u|^2 + W(u) \right) \dd x,
\]
where $\Omega$ is an open and bounded subset of $\R^n$, and $A \subset \Omega$ is an open, bounded, connected, and Lipschitz set with $\partial A \cap \Omega \not= \varnothing$.
Suppose that
\begin{enumerate}
\item $u \in W^{1,2}(\Omega ; \R^m) \cap L^{\infty}(\Omega ; \R^m)$,\medskip
\item $|u(x)-a| \leq r$ on $\partial A \cap \Omega$ (in the sense of the trace) for some $r$ with $2r \in (0, r_0]$,\medskip
\item $| \{  x \in A \mid |u(x)-a|>r \}| >0$ (when applied to sets, the notation $|\cdot|$ stands for the Lebesgue measure).
\end{enumerate}
Then, there exists
$\tilde u \in W^{1,2}(\Omega ; \R^m) \cap L^{\infty}(\Omega ; \R^m)$ such that
\[
\begin{cases}
\tilde u (x)=u(x), &\text{ on } \Omega \setminus A,\medskip\\
| \tilde u (x) -a | \leq r, &\text{ on } A, \medskip \\
J_{\Omega} ( \tilde u) < J_{\Omega} (u).
\end{cases} \]
\end{clemma}

\begin{remarks}
The hypothesis \eqref{hypoth} is a very mild non degeneracy hypothesis for the minimum $a$. Notice that it allows $C^\infty$ contact. This can be useful in applying the lemma to certain situations where degeneracy is natural (see \cite{ball-crooks}).

The proof utilizes variations (replacements) of the map $u$ which are obtained by deforming the radial part but keeping the angular part fixed. The previous subsection guarantees that these variations are in $W^{1,2} (\Omega, \R^m)$ and provides a convenient formula for calculating their energy. Note that the replacements are not necessarily local since $|u(x) - a|$ is not \emph{a priori} restricted.

The way the lemma is implemented is as follows: If $u$ is a minimizer, then $|u(x) - a| \leq r$ on $A$. The idea is that $u(x)$ cannot make an excursion far away from $a$ and benefit by entering a low energy region of $W$ because the energy required for getting outside a small neighborhood of $a$ exceeds the energy needed to bring it down to $a$ within that small neighborhood and keep it there.

Hypothesis (ii) is the most difficult to verify. An \emph{a priori} bound on the energy $J$ together with an \emph{a priori} $L^\infty$ bound on $|\nabla u|$ allow in certain situations the construction of such sets $A$.
\end{remarks}

\begin{proof}
We utilize the polar representation in the first subsection above.
\subsubsection*{Step 1} We begin by settling the lemma under the additional hypothesis
\begin{equation}\label{add hyp}
\rho (x) \leq 2r \leq r_0 \text{ a.e.\  in } A.
\end{equation}
Set
\begin{equation}\label{set-u}
\tilde u (x)=
\begin{cases}
a, &\text{ for } x \in A_0=\{ x \in A \mid \rho =0\}, \medskip\\
a+ \min \{ \rho (x), r \} \nu(x), &\text{ for } x \in A_+ =\{ x \in A \mid \rho >0 \} \medskip \\
u(x), &\text{ for } x \in \Omega \setminus A.
\end{cases}
\end{equation}
By Corollary \ref{cor3}, and since $\tilde{u} =u$ on $\partial A \cap \Omega$, we have that $\tilde u \in W^{1,2}(\Omega ; \R^m) \cap L^{\infty}(\Omega ; \R^m)$. Thus,
\begin{equation}\label{int1}
 \int_{\Omega} |\nabla \tilde u|^2 \dd x = \int_{A} |\nabla \tilde u|^2  \dd x + \int_{\Omega \setminus A} |\nabla u|^2 \dd x.
\end{equation}
On the other hand, via \eqref{integral rho tilde cor},
\begin{align}\label{int2}
 \int_{A}|\nabla \tilde u|^2 \dd x &=
 \int_{A} |\nabla\tilde \rho|^2 \dd x + \int_{A_+} \tilde \rho^2 \sum_j  \langle \nu^{,j} , \nu^{,j} \rangle \dd x  \nonumber\\
&\leq \int_{A} |\nabla  \rho|^2 \dd x + \int_{A_+}  \rho^2 \sum_j  \langle \nu^{,j} , \nu^{,j} \rangle \dd x \nonumber\\
&=  \int_{A}  |\nabla u |^2 \dd x.
\end{align}
Then, \eqref{int1} and \eqref{int2} give
\begin{equation}\label{int3}
 \int_{\Omega} |\nabla \tilde u|^2 \dd x \leq  \int_{\Omega} |\nabla u|^2 \dd x.
\end{equation}
Next we treat the potential term in $J$. We write
$$\int_{\Omega} W(\tilde u(x)) \dd x = \int_{A} W(\tilde u(x))  \dd x + \int_{\Omega \setminus A}  W( u(x)) \dd x,$$
and we have
\begin{align}\label{int4}
\int_{A} W(\tilde u(x))  \dd x&= \int_{A} W(a+\tilde \rho (x) \nu(x)) \dd x \nonumber\\
&= \int_{A \cap \{ \rho \leq r \}} W(a+\tilde \rho (x) \nu(x)) \dd x +  \int_{A \cap \{ \rho > r \}} W(a+\tilde \rho (x) \nu(x)) \dd x    \nonumber \\
&= \int_{A \cap \{ \rho \leq r \}} W(u(x)) \dd x +  \int_{A \cap \{ \rho > r \}} W(a+\tilde \rho (x) \nu(x)) \dd x.
\end{align}
Since by \eqref{hypoth}, \eqref{add hyp} and (iii) above there holds
\begin{align}\label{int5}
\int_{A \cap \{ \rho > r \}} W(a+\tilde \rho (x) \nu(x)) \dd x &<\int_{A \cap \{ \rho > r \}} W(a+ \rho (x) \nu(x)) \dd x\nonumber\\
&=\int_{A \cap \{ \rho > r \}} W(u(x)) \dd x,
\end{align}
and so by \eqref{int4}, \eqref{int5},
\[ \int_{A } W( \tilde u(x)) \dd x  <  \int_{A } W(u(x)) \dd x. \]
Thus, the lemma is established under \eqref{add hyp}.

\subsubsection*{Step 2}
Now we can assume that \eqref{add hyp} does not hold, hence
\[ | \{  x \in A \mid \rho(x)>2r \}| >0, \]
where $|\cdot|$ is the Lebesgue measure. Set (cf.\ the Remark following the proof of Corollary \ref{cor3})
\[
 \tilde u (x)=
\begin{cases}
a, &\text{ for } x \in \tilde A_0 \medskip \\
a+ \min \{ \rho (x), r \} \alpha(\rho(x)) \nu(x), &\text{ for } x \in \tilde A_+ \medskip \\
u(x), &\text{ for } x \in \Omega \setminus A,
\end{cases}
\]
where $\alpha$ is defined in the aforementioned remark and $\tilde \rho (x)= | \tilde u (x)-a|$, $\tilde A_0=\{ x \in A \mid \tilde \rho (x)=0 \}$, and $\tilde A_+=\{ x \in A \mid \tilde \rho (x) >0 \}$. We note that as in \eqref{set-u}, we conclude that $\tilde u \in W^{1,2}(A ; \R^m) \cap L^{\infty}(A ; \R^m)$.

Moreover,
\[ \int_{\Omega} |\nabla \tilde u|^2 \dd x = \int_{A} |\nabla \tilde u|^2  \dd x + \int_{\Omega \setminus A} |\nabla u|^2 \dd x, \]
and
\[
\int_{A}|\nabla \tilde u|^2 \dd x = \int_{A} |\nabla\tilde \rho|^2 \dd x + \int_{ \tilde A_+} \tilde \rho^2 \sum_j  \langle \nu^{,j} , \nu^{,j} \rangle \dd x. \]
Note that $\tilde \rho (x)=  \min \{ \rho (x), r \} \alpha(\rho(x))$ on $\tilde A_+$ and $\tilde \rho (x)=r \alpha(\rho(x))$ on $\tilde A_+ \cap \{ r \leq \rho \leq 2r  \}$.
We have
\begin{equation}\label{int3b}
\int_{A}|\nabla \tilde \rho|^2 \dd x = \int_{\tilde A_+} |\nabla\tilde \rho|^2 \dd x=\int_{\tilde A_+ \cap \{ r \leq \rho \leq 2r \} }|\nabla \tilde \rho|^2 \dd x + \int_{\tilde A_+ \cap \{  \rho <r\}} |\nabla \rho|^2 \dd x.
\end{equation}

On $\tilde A_+ \cap \{ r \leq \rho \leq 2r  \}$ there holds
\[ |\nabla \tilde \rho (x)|^2 =|r \alpha'(\rho) \nabla \rho (x)|^2 \leq |\nabla  \rho (x)|^2, \]
hence,
\begin{align}\label{int5b}
\int_{ \tilde A_+\cap \{ r \leq \rho \leq 2r  \} } |\nabla & \tilde \rho|^2 \dd x  +  \int_{\tilde A_+ \cap \{ \rho < r \} }|\nabla  \rho|^2 \dd x   \nonumber \\
&\leq
\int_{ \tilde A_+\cap \{ r \leq \rho \leq 2r  \} }|\nabla \rho|^2 \dd x +  \int_{\tilde A_+ \cap \{ \rho < r \} }|\nabla  \rho|^2 \dd x  \nonumber \\
&\leq \int_{\tilde A_+} |\nabla \rho|^2 \dd x,
\end{align}
and therefore by \eqref{int3b}, \eqref{int5b} we have
\[ \int_{A} |\nabla \tilde \rho|^2 \dd x \leq  \int_{A } |\nabla \rho|^2 \dd x, \]
and since $\tilde \rho \leq \rho$,
\[ \int_{A} |\nabla \tilde u|^2 \dd x \leq  \int_{A } |\nabla u|^2 \dd x. \]

Next, we consider the potential on $A \cap \{ r \leq \rho \leq 2r  \}$. We have
\begin{align}\label{int7b}
W(\tilde u(x)) &= W(a+r \alpha(\rho(x)) \nu(x)) \nonumber \\
&\leq W(a+r \nu(x))   \nonumber \\
&\leq  W(a+\rho (x) \nu(x)) \nonumber \\
&=W(u(x)), \end{align}
where \eqref{hypoth} was utilized in the last two inequalities. Note that
\[ W(\tilde u(x)) =W(u(x)) \text{ on } A \cap \{ \rho <r \}. \]
By examining the inequalities above we observe that
\[ J( \tilde u) <J(u) \]
will follow if any of the following strict inequalities holds:
\begin{enumerate}
\item $|  A \cap \{ r< \rho \leq 2r \}| >0$ (by \eqref{int7b}, \eqref{hypoth}),\medskip
\item $\displaystyle{\int_{A \cap \{  \rho > 2r\}} |\nabla \rho|^2 \dd x >0}$,\medskip
\item $\displaystyle{\int_{A \cap \{  \rho > 2r\}} W(u) \dd x >0}$.
\end{enumerate}
Suppose for the sake of contradiction that the first inequality is violated, that is,
\[ |A \cap \{ r< \rho \leq 2r \}| =0. \]
From the operating hypothesis in Step 2, we have
\[ |A \cap \{  \rho >2r \}| >0, \]
Let us partition $A$ into the three sets
\[ E_1:=  A \cap \{\rho \leq r \},\ E_2:=  A \cap \{ r< \rho \leq 2r \},\ E_3:=  A \cap \{  \rho >2r \}. \]
According to what precedes, there holds $|E_2|=0$ and $|E_3|>0$. We consider now\footnote{We thank Panayotis Smyrnelis for simplifying significantly our previous lengthy argument.} the following Sobolev functions defined on the open and connected set $A$. Let
\[
 \sigma(x)=\min \{ \rho(x), 2r  \}=
\begin{cases}
\rho(x), &\text{ for } x \in E_1, \medskip \\
2r, &\text{ for } x \in E_3,
\end{cases}
\]
and
\[
 \tau(x)=\max \{ \sigma, r  \}-r=
\begin{cases}
0, &\text{ for } x \in E_1, \medskip \\
r, &\text{ for } x \in  E_3.
\end{cases}
\]
Since $\tau$ is equal almost everywhere to a multiple of the characteristic function of $E_3$ and since it is a Sobolev function, it follows that $\nabla \tau =0$ a.e.\ in $A$. As a consequence of connectedness, $\tau = r$ a.e.\ in $A$ (cf. \cite[p.~307]{evans}), $|E_1|=0$, and $\rho>2r$ a.e.\ in $A$. This is a contradiction since we have assumed that $\rho \leq r$ on $\partial A \cap \Omega$ in the sense of the trace. The proof of the lemma is complete.
\end{proof}

\section{The constrained variational problem}\label{section3}
For fixed $N\geq 1$ consider the set of maps $X_N$ defined by
\[ X_N:=\left\{ u\in W_{\rm loc }^{1,2}(\Omega;\R^m) ~\Big|~ \vert u((s,y))-a_\pm\vert\leq\frac{r_0}{2} ,\text{ for } \pm s\geq NL \right\}, \]
where $r_0$ is the constant in Hypothesis \ref{h2}.
We will minimize the energy $J$ in the class $X_N$. Notice that the constant maps $u\equiv a_-$, $u\equiv a_+$ are not allowed in $X_N$.
Existence of minimizers of $J_{\Omega}$ in $X_N$ is rather standard but we will give the details for the convenience of the reader. We remark that due to the presence of the constraint we can not claim \emph{a priori} that minimizers satisfy the Euler--Lagrange equation.

\begin{proposition}\label{minimize}
Assume that $W:\R^m\to\R$ is of class $C^2$ and satisfies Hypothesis \ref{h3}.
Then, there exists $u^N\in X_N$ such that
\[ J_{\Omega}(u^N)=\min_{X_N}J_{\Omega}(u). \]
\end{proposition}

\begin{proof}
 Define the affine map
\[
\bar{u}((s,y))=
\begin{cases}
a_-, &\text{for } s \leq L,\medskip\\
\dfrac{1-s/L}{2}a_- + \dfrac{1+s/L}{2}a_+, &\text{for } s \in(-L,L),\medskip\\
a_+, &\text{for } s\geq L.
\end{cases}
\]
Clearly $\tilde{u} \in X_N$ and $J_{\Omega}(\tilde{u})<+\infty$. Thus 
\begin{equation}\label{upper-lower}
0\leq\inf_{X_N}J_{\Omega}(u)\leq J_{\Omega}(\bar{u})<+\infty.
\end{equation}
Given $u \in X_N$ that satisfies $J_{\Omega}(u)\leq J_{\Omega}(\tilde{u})$ set $u_M=0$ if $u=0$ and $u_M=\min\{\vert u\vert, M\}u/\vert u\vert$ otherwise. Then Hypothesis \ref{h3} implies
\[J_{\Omega}(u_M)=\int_{\{\vert u\vert<M\}}\left( \frac{1}{2} \vert \nabla u \vert^2 + W(u) \right) \!\dd x +\int_{\{\vert u\vert\geq M\}}W(\frac{M}{\vert u\vert}u)\!\dd x \leq  J_{\Omega}(u).\]
 It follows that  we can restrict to $X_N\cap\{\|u\|_{L^\infty(\Omega;\R^m)}\leq M\}$ and therefore we may assume that $W(u)\geq c^2\vert u\vert^2$ for $\vert u\vert\geq M+1$ for some $c>0$.
Let $\{u_j\}_{j=1}^\infty\subset X_N$ be a minimizing sequence. From (\ref{upper-lower}) we have
\[ \int_\Omega\frac{1}{2}\vert\nabla u\vert^2 \dd x\leq J_{\Omega}(u_j)\leq J_{\Omega}
(\bar{u}). \]
Hence, using also that $\|u_j\|_{L^\infty(\Omega;\R^m)}\leq M$, we have that, possibly by passing to a subsequence,
\[  u_j\rightharpoonup u^N, \text{ in } W^{1,2}_{\text{loc}}(\Omega;\R^m), \]
by weak compactness.
By compactness of the embedding we can assume that $u_j\to u^N$ strongly in $L^{2}_{\text{loc}}(\Omega;\R^m)$ and therefore, along a further subsequence,
\[ \lim_{j\to+\infty}u_j((s,y))=u^N((s,y)), \text{ a.e.\ in } \Omega. \]
Weak lower semi-continuity of the $L^2$ norm gives
\[ \liminf_{j\to+\infty}\int_\Omega\frac{1}{2}\vert\nabla u_j\vert^2 \dd x\geq\int_\Omega\frac{1}{2}\vert\nabla u^N\vert^2 \dd x,
\]
and by Fatou's lemma,
\[ \liminf_{j\to+\infty}\int_\Omega W(u_j) \dd x\geq\int_\Omega W(u^N) \dd x. \]
The proof is complete.
\end{proof}

\section{Removing the constraint}\label{section4}

\subsection{Removing the constraint for \texorpdfstring{$s\in(-\infty,-N L)\cup(N L,+\infty)$}{s in (-infinity,NL) union (NL,+infinity)}}
\label{section4a}
Removing the constraint in the interior of the cylinders
\[ \left\{ |u-a_{\pm}| \leq \frac{r_0}{2}, \text{ for } \pm s \geq NL \right\}, \]
is easier since linearization about the minima $a_\pm$ is available. So, in this subsection the Cut-Off Lemma is \emph{not} utilized.

Let $g:[0,r_0]\to\R$ be the function defined by
\[ g(r) := \min_{r\leq r^\prime\leq r_0} \min_{\substack{\nu\in \mathbb{S}^{m-1} \\ a\in\{a_-,a_+\} }} \langle W_u(a+r^\prime \nu),\nu\rangle,\text{ for } r\in[0,r_0]. \]
From Hypotheses \ref{h1} and \ref{h2} we have that $g(0)=0$ and that $g$ is strictly increasing. Let $f:[0,r_0]\to [0,+\infty]$ be a strictly increasing function that satisfies $f(0)=0$ and
\begin{equation}\label{fg-ineq}
0 \leq f(r^2)\leq 2rg(r),\text{ for } r\in[0,r_0].
\end{equation}
Observe that if $a_\pm$ is nondegenerate, then $g$ is bounded below by a linear map. Therefore, in that case we can assume that
\[f(t)=c^2t,\text{ for } t\in[0,r_0^2],\]
for some constant $c>0$.

The reason for introducing the function $f$ as in \eqref{fg-ineq} will become apparent in Lemma \ref{rho-variation}.

\begin{lemma}\label{lemma-phi}
Let $f$ be as in \eqref{fg-ineq} and let $\omega = \cup_{s\in(-L,L)}\Omega^s$. For $t\in(0,r_0^2]$ let $\varphi:\omega\times(0,r_0^2] \to \R$ be the solution of the problem
\begin{equation}\label{system-phi}
\begin{cases}
\Delta\varphi=f(\varphi), &\text{in } \omega,\medskip\\
\varphi=t, &\text{on } \partial^b\omega,\medskip\\
\dfrac{\partial\varphi}{\partial n}=0, &\text{on } \partial^l\omega,
\end{cases}
\end{equation}
where $\partial^b\omega := \partial\omega\cap \left( \{-L,L\}\times\R^{n-1} \right)$, $\partial^l\omega=\partial\omega\setminus\partial^b\omega$ is the lateral boundary of $\omega$, and $n$ is the outward normal.

Then, the following hold.
\begin{enumerate}
\item $\varphi((s,y),t) < t$, for $(s,y) \in \omega$, and $\hat{t} := \max_{(0,y) \in \overline{\Omega^0}} \varphi((0,y),t) < t.$\medskip
\item $\lim_{j\to+\infty}t_j = 0$, where $\{t_j\}$ is defined by $t_0=t$, $t_j=\hat{t}_{j-1}, \text{ for } j = 1,\dots$\medskip
\item If $f$ is linear, that is, $f(t)=c^2t$ for some $c>0$, then there is a $\theta\in(0,1)$ such that
\[ t_j = \theta^jt, \text{ for } j = 1,\dots\]
\end{enumerate}
\end{lemma}

\begin{proof}With the change of variables $\varphi=t+\psi$, problem \eqref{system-phi} becomes
\begin{equation}\label{system-psi}
\begin{cases}
\Delta\psi = \tilde{f}(\psi) := f(t+\psi), &\text{in } \omega,\medskip\\
\psi=0, &\text{on } \partial^b\omega,\medskip\\
\dfrac{\partial\psi}{\partial n}=0, &\text{on } \partial^l\omega,
\end{cases}
\end{equation}
Let $W_\sharp^{1,2} (\omega)$ be the closure of the set $\{\psi \in C^\infty(\bar{\omega}) \mid \psi^+=0 \text{ on } \partial^b\omega\}$ in $W^{1,2} (\omega)$. We can assume that $f$ is extended to a nondecreasing nonnegative function $f: \R \to [0,+\infty]$. Since $\tilde{f}(-t) = f(0) = 0$, the function $\underline{\psi} \equiv -t$ is a weak subsolution of equation \eqref{system-psi}, that is,
\[
\int_\omega \left( \nabla \underline{\psi} \nabla z + \tilde{f}(\underline{\psi}) z \right)\! \dd x \leq 0, \text{ for } z \in W_\sharp^{1,2} (\omega) \text{ with } z \geq 0.
\]
Similarly, the fact that $f$ is nonnegative implies that $\overline{\psi}\equiv 0$ is a weak supersolution of \eqref{system-psi}. Moreover,
\begin{equation}\label{on-boundary}
\underline{\psi} |_{\partial^b\omega} < 0,\ \overline{\psi}_{\partial^b\omega}=0,\text{ and }
\underline{\psi}<\overline{\psi}, \text{ a.e.\ in } \omega.
\end{equation}
The existence of weak sub and supersolutions $\underline{\psi}$ and $\overline{\psi}$ that satisfy \eqref{on-boundary} imply the existence of a weak solution $\psi \in W_\sharp^{1,2} (\omega)$ of \eqref{system-psi} such that
\[ \underline{\psi}\leq\psi\leq\overline{\psi}, \text{ a.e.\ on } \omega. \]
This can be proved as in \cite[p.~543]{evans}. From elliptic regularity $\psi$ is a $C^2$ map away from $\partial\Omega \cap \left( \{\pm L\}\times\R^{n-1} \right)$. Therefore, the Hopf boundary lemma and the strong maximum principle imply $\psi<0$, in $\omega\cup\partial^l\omega$ and, therefore, (i) is established.

The sequence $\{t_j\}$ is monotone decreasing and bounded below, therefore the limit in (ii) exists. Assume that $\lim_{j\to+\infty}t_j=t_\infty>0$ and let $\varphi_j$ the solution of (\ref{system-phi}) corresponding to $t_j$. Since $f$ is increasing, the difference $w:=\varphi_j-\varphi_{j+1}$ satisfies the linear equation
\[ \begin{cases}
\Delta w - c^2 w = 0, &\text{in }\omega,\medskip\\
w = t_j - t_{j+1} > 0, &\text{on } \partial^b\omega,\medskip\\
\dfrac{\partial w}{\partial n} = 0, &\text{on } \partial^l\omega
\end{cases} \]
where
\[ c^2 =
\begin{cases}
\dfrac{f(\varphi_{j}) - f(\varphi_{j+1})}{\varphi_{j} - \varphi_{j+1}}, &\text{if } \varphi_{j} - \varphi_{j+1} \neq 0,\medskip\\
f^\prime(\varphi_{j+1}), &\text{if } \varphi_{j} - \varphi_{j+1} = 0,
\end{cases} \]
therefore the comparison principle in $W_\sharp^{1,2} (\omega)$ implies that $w \geq 0$ in $\omega$ and we have that $\varphi_{j+1} \leq \varphi_j$. Since the sequence $\{\varphi_j\}$ of continuous functions in $\omega$ is bounded, we conclude that, as $j\to+\infty$, $\varphi_j$ converges uniformly to a map $\varphi_\infty$. Actually, $\varphi_\infty\in W_\sharp^{1,2} (\omega)$. To see this we note that from the fact that $f$ is bounded and that the sequence $\{\varphi_j\}$ is bounded, it follows that also $\|f(\varphi_j)\|_{L^2}$ is uniformly bounded. This and the fact that $\varphi_j$ is a weak solution of \eqref{system-phi} imply a uniform bound for $\|\varphi_j\|_{W_\sharp^{1,2}} (\omega)$. It follows that $\varphi_\infty\in W_\sharp^{1,2} (\omega)$ as a weak limit of the sequence $\{\varphi_j\}$ in $W_\sharp^{1,2}(\omega)$ and that $\varphi_\infty$ is a weak solution of \eqref{system-phi}. By elliptic regularity, $\varphi_\infty$ is $C^2$ away from $\partial\Omega \cap \left( \{\pm L\}\times\R^{n-1} \right)$. Therefore, uniform convergence of $\varphi_j$ to $\varphi_\infty$ implies that
\[ t_\infty = \lim_{j\to+\infty} t_j = \lim_{j\to+\infty} \max_{\overline{\Omega^0}} \varphi_j = \max_{\overline{\Omega^0}} \varphi_\infty.
\]
Hence, the strong maximum principle yields $\varphi_\infty\equiv t_\infty$ but this and $f(t_\infty)>0$ contradict \eqref{system-phi}. This contradiction establishes (ii).

To prove (iii) we note that if $f$ is linear, system \eqref{system-phi} is also linear and therefore $\varphi (\cdot,t) = t \varphi (\cdot,1)$. This implies $\hat{t} = \hat{t}_1$ and therefore we can take $\theta = \hat{t}_1 < 1$ and $t_j = \theta^j$.
\end{proof}

\begin{lemma}\label{rho-variation}
Let $u^N$ be a minimizer as in Proposition \ref{minimize}. Set
\[ \rho = \vert u^N - a_+ \vert \]
and let $\omega_k = \cup_{s\in(-L,L)} \Omega^{s+(N+k)L}, \text{ for } k = 1,\dots$ Then,
\[ \int_{\omega_k}\langle\nabla(\rho^2),\nabla p\rangle+f(\rho^2)p\leq 0, \]
for all $p\geq 0$ in $W^{1,2}(\omega_k) \cap L^\infty(\omega_k)$ such that $p=0$ on $\partial^b\omega_k$.
\end{lemma}

\begin{proof}
For $p\in W^{1,2}(\omega_k)\cap L^\infty(\omega_k)$ as above and for $\varepsilon>0$ small let $u_\varepsilon$ be the variation of $u^N$ defined by
\[
u_\varepsilon=
\begin{cases}
u^N - \varepsilon p \rho \nu = a_+ + (1-\varepsilon p)\rho \nu = a_+ + (1-\varepsilon p) (u^N - a_+), &\text{on } \omega_k,\medskip\\
u^N, &\text{on } \Omega\setminus\overline{\omega_k},
\end{cases}
\]
where $\nu = (u^N - a_+) / |u^N - a_+|$. Note that for $\varepsilon>0$ sufficiently small we have $0\leq 1-\varepsilon p\leq 1$ and therefore $u_\varepsilon$ satisfies the constraint $\vert u_\varepsilon-a_+\vert=(1-\varepsilon p)\rho\leq 2r$ on $\omega_k$. This and the minimality of $u^N$ imply
\begin{align*}
0 &\leq \frac{d}{d_\varepsilon}\Big|_{\varepsilon=0} J_{\omega_k} (u_\varepsilon) \\
&=\frac{d}{d_\varepsilon}\Big|_{\varepsilon=0}\int_{\omega_k}
\frac{1}{2}\Bigg( \vert \nabla ((1-\varepsilon p)\rho) \vert^2 + ( (1-\varepsilon p)\rho)^2 \sum_j \langle \nu^{,j},\nu^{,j} \rangle \Bigg) + W((1-\varepsilon p)\rho \nu),
\end{align*}
where we have also used the polar form \eqref{integral polar} of the energy. It follows that
\begin{multline*}
-\int_{\omega_k} \Bigg( \langle \nabla\rho, \nabla(p\rho) \rangle + p\rho \sum_j \langle \nu^{,j},\nu^{,j} \rangle + p\rho W_u(\rho \nu)\nu \Bigg) \\
=-\int_{\omega_k}
\Bigg( \frac{1}{2} \left\langle \nabla(\rho^2),\nabla p \right\rangle + p(\vert \nabla\rho \vert^2 + \rho \sum_j \langle \nu^{,j},\nu^{,j} \rangle ) + p\rho W_u(\rho \nu)\nu \Bigg) \geq 0
\end{multline*}
and, since $p ( \vert \nabla\rho \vert^2 + \rho \sum_j \langle \nu^{,j},\nu^{,j} \rangle ) \geq 0$ and, by the definition of $f$, there holds that $2\rho \langle W_u(\rho \nu,)\nu\rangle\geq f(\rho^2)$, we have
\[ -\int_{\omega_k} \frac{1}{2} \left( \left\langle\nabla(\rho^2),\nabla p \right\rangle + f(\rho^2)p \right) \geq 0.
\qedhere
\]
\end{proof}

\begin{remark} At first sight, the more natural variation would be $u_\varepsilon=u^N +\varepsilon p \nu$, which formally leads to $\Delta\rho\geq g(\rho)$. The problem with this is that $u_\varepsilon$ does not, in general, vanish when $\rho$ vanishes and therefore $u_\varepsilon$ may not be a $W^{1,2}$ map.
\end{remark}

\begin{lemma}\label{comparison}
Let $k=1,\dots$ be given and assume that $\rho^2\leq t$ on $\partial^b\omega_k$. Then,
\[ \rho^2((s,y)) \leq \varphi((s-(N+k)L,y),t) < t, \text{ for } (s,y)\in\omega_k, \]
and
\[ \rho^2 \leq \hat{t}, \text{ on } \Omega^{(N+k)L}, \text{ for } k = 1,\dots,\]
with $\rho$ as in Lemma \ref{rho-variation}.
\end{lemma}

\begin{proof}
Set $\varphi_k((s,y),t) := \varphi((s-(N+k)L,y),t)$, for $(s,y)\in\omega_k$; then, $\varphi_k$ satisfy all the statements in Lemma \ref{lemma-phi} with $\omega$ replaced by $\omega_k$. Therefore, from \eqref{system-phi} and integration by parts we get
\begin{equation}\label{weak-system-phi}
-\int_{\omega_k} \frac{1}{2} \left( \langle \nabla \varphi_k,\nabla p\rangle + f(\varphi_k)p \right) = 0,
\end{equation}
for all $p\in W^{1,2}(\omega_k) \cap L^\infty(\omega_k)$ such that $p=0$ on $\partial^b\Omega_k$. From (\ref{weak-system-phi}) and Lemma \ref{rho-variation} it follows that
\begin{equation}\label{difference}
\int_{\omega_k} \left\langle \nabla(\rho^2 - \varphi_k),\nabla p \right\rangle + (f(\rho^2)-f(\varphi_k))p\leq 0,
\end{equation}
for all $p\geq 0$ in $W^{1,2}(\omega_k)\cap L^\infty(\omega_k)$ such that $p=0$, on $\partial^b\omega_k$. In particular, for $p=(\rho^2-\varphi_k)^+$, \eqref{difference} yields
\begin{eqnarray}\label{difference-1}
\int_{\omega_k\cap\{\rho^2>\varphi_k\}} \vert\nabla(\rho^2 - \varphi_k)^+ \vert^2 + (f(\rho^2)-f(\varphi_k))(\rho^2-\varphi_k)^+\leq 0.
\end{eqnarray}
Since $f$ is strictly increasing, we have $f(\rho^2) - f(\varphi_k) > 0$ for $\rho^2>\varphi_k$ and therefore \eqref{difference-1} implies $\rho\leq\varphi_k$, a.e.\ on $\omega_k$.
This and Lemma \ref{lemma-phi} conclude the proof.
\end{proof}

\begin{proposition}\label{convergence}
Let $t_0 = \frac{r_{0}^{2}}{4}$ and $t_j = \hat{t}_{j-1}$, for $j=1,\dots$ Then,
\[
\begin{cases}
\rho^2 \leq t_j, &\text{ on } \Omega^s, \text{ for } s > (N+j)L, \ j=1,\dots\medskip\\
\rho^2 < t_{j-1}, &\text{ on } \Omega^s, \text{ for } s > (N+j-1)L, \ j=1,\dots
\end{cases}
\]
with $\rho$ as in Lemma \ref{rho-variation}.
\end{proposition}

\begin{proof}
Since $u^N$ satisfies the constraint, we have $\rho^2\leq t_0$, on $\Omega^s$, for $s\geq NL$. Therefore, Lemma \ref{comparison} with $t=t_0$ and $k=1,\dots$ yields
\[ \rho^2\leq t_1, \text{ on } \Omega^{(N+k)L}, \text{ for } k=1,\dots \]
This and Lemma \ref{comparison} with $t=t_0$ and $k=2,\dots$ imply
\[\rho^2\leq t_1, \text{ on } \Omega^s, \text{ for } s \geq (N+1)L \]
and
\[ \rho^2\leq t_2, \text{ on } \Omega^{(N+k)L}, \text{ for } k=2,\dots\]
Induction on $j$ concludes the proof of the first inequality. The second inequality follows from the first and from Lemma \ref{comparison}, which imply
\[ \rho^2 < t_{j-1}, \text{ on } \omega_j,\ j=1,\dots \qedhere \]
\end{proof}

Obviously Proposition \ref{convergence} implies
\begin{equation}\label{lim-piu}
\lim_{\substack{s\to+\infty\\ (s,y)\in\Omega}} \rho(s,y) = a_+
\end{equation}
and, by statement (iii) of Lemma \ref{lemma-phi}, if $a_+$ is nondegenerate then
\[\vert u(s,y) - a_+ \vert \leq K_0 \e^{-k_0 s}, \text{ for }s>0, \text{ with } (s,y)\in\Omega .\]
The analogous statements concerning $a_-$ are proved in a similar way.

Thus, in this subsection we established that $u^N$, for $N \geq 1$, does not realize the constraint in $(-\infty, NL) \cup (NL,+\infty)$ and hence satisfies in this set the equation, the Neumann condition, and also the asymptotic condition in \eqref{system-boundary}, which takes the form \eqref{heteroclinic-ode} for nondegenerate $a_\pm$.

\subsection{Removing the constraint at \texorpdfstring{$s=\pm NL$}{s=+NL or s=-NL}.}\label{section4b}
In this part of the proof we use the Cut-Off Lemma developed in Section \ref{section2}.
The minimizer $u^N$ is a classical solution of (\ref{system}) in $\Omega_N:=\cup_{s\in(-NL,NL)}\Omega^s$. Moreover, from Hypothesis \ref{h3} it follows that
\[ \|u^N\|_{L^\infty(\Omega_N;\R^m)} \leq M,\]
with $M$ independent of $N$. Therefore, linear elliptic theory implies that
\begin{equation}\label{nabla-bound}
\vert\nabla u^N\vert\leq M^\prime, \text{ on } \overline{\Omega_{N-1}},
\end{equation}
for some $M^\prime>0$ independent of $N\geq 2$, since in $\overline{\Omega_{N-1}}$ the function $u^N$ satisfies the equation in \eqref{system-boundary}.

\begin{lemma}\label{lower-potential-bound}
Let $r \in (0,\frac{r_0}{2})$ be fixed. Then there exist $w_0>0$ and $\delta\in(0,\frac{L}{2})$ such that
\[ (\bar{s},\bar{y})\in\Omega_{N-2} \text{ and } \min_{a \in \{a_-,a_+\}}\vert u^N((\bar{s},\bar{y})) - a\vert\geq r, \]
imply
\[ J_{\Omega_\delta^{\bar{s}}}(u^N)\geq w_0, \]
where $\Omega_\delta^{\bar{s}}:=\cup_{s\in(\bar{s}-\delta,\bar{s}+\delta)}\Omega^s$.
\end{lemma}
\begin{proof}
Since $\Omega$ is periodic and of class $C^2$, it satisfies the interior sphere condition with a ball of fixed radius. From this it follows that there is $\delta>0$ such that each $x\in\Omega$ belongs to $B_{x^\prime,\frac{\delta}{2}}$ for some $x^\prime\in\Omega$.

Assume that there exist $(\bar{s},\bar{y})\in\Omega_{N-2}$ such that
\[\min_{a \in \{a_-,a_+\}}\vert u^N((\bar{s},\bar{y}))-a\vert\geq r.\]
Observe that if we take $\delta<\frac{L}{2}$, we have
\[B_{(\bar{s},\bar{y})^\prime, \frac{\delta}{2}}\subset\Omega_{N-1}, \]
and we can apply (\ref{nabla-bound}). Therefore, if we restrict the choice of $\delta>0$ to $\delta<\min\{\frac{r}{2M^\prime},\frac{L}{2}\}$,
the bound  (\ref{nabla-bound}) implies that
\[\min_{a \in \{a_-,a_+\}}\vert u^N((s,y)) - a\vert > \frac{r}{2}, \text{ for } (s,y) \in B_{(\bar{s},\bar{y})^\prime, \frac{\delta}{2}}.\]
From this and from the properties of $W$ it follows that
\[
\int_{B_{(\bar{s},\bar{y})^\prime, \frac{\delta}{2}}}W(u^N)\geq w_0,
\]
for some $w_0>0$. Therefore,
\[J_{\Omega_\delta^{\bar{s}}}(u^N)\geq\int_{\Omega_\delta^{\bar{s}}}W(u^N)\geq
\int_{B_{(\bar{s},\bar{y})^\prime, \frac{\delta}{2}}}W(u^N)\geq w_0,\]
which concludes the proof of the lemma.
\end{proof}

\begin{proof}[Conclusion of the proof of Theorem \ref{teo-1}]
Assume $\delta \in (0,\frac{L}{4})$ in Lemma \ref{lower-potential-bound} and observe that then we have
\begin{equation}\label{omega-hk}
\Omega^{hL}_{\delta} \cap \Omega^{kL}_{\delta} = \varnothing, \text{ for } h \neq k,\text{ with } h,k \in [-(N-2),N-2].
\end{equation}
Let $Z\leq 2N-3$ be the number of integers $h\in[-(N-2),N-2]$ such that
\[ \min_{a \in \{a_-,a_+\}}\vert u^N((h L,y)) - a\vert\geq r,\; \text{ for some }\;  (h L,y)\in\Omega^{h L}.\]
From Lemma \ref{lower-potential-bound}, the \emph{a priori} estimate \eqref{upper-lower}, and \eqref{omega-hk} we have
\[ Z \leq \frac{J_\Omega(\bar{u})}{w_0},\]
therefore $2N-3 > \frac{J_\Omega(\bar{u})}{w_0}$ is a sufficient condition for the existence of $\bar{h} \in [-(N-2),N-2]$ such that
\[\min_{a \in \{a_-,a_+\}}\vert u^N((\bar{h}L,y)) - a\vert <  r,\;\text{ for all }\;(\bar{h}L,y)\in\Omega^{\bar{h}L}. \]
Since by Hypothesis \ref{h5} the domain $\Omega^{\bar{h}L}$ is connected and $u^N$ is smooth in $\Omega_{N-1}$, there exists $a\in\{a_-,a_+\}$ such that
\[\vert u^N((\bar{h}L,y))-a\vert< r, \text{ for all } (\bar{h}L,y)\in\Omega^{\bar{h}L}.\]
Assume for definiteness that $a=a_+$ (if $a=a_-$ the argument is completely analogous) and use (\ref{lim-piu}) to fix $\bar{s}>(N+2)L$ such that
\[ \vert u^N((\bar{s},y))-a_+\vert< r , \text{ for all } (\bar{s},y)\in\Omega^{\bar{s}}.\]
Then, the minimality of $u^N$ and the Cut-Off Lemma imply $\vert u^N((s,y))-a_+\vert \leq r< \frac{r_0}{2}$, for all $s\in[\bar{h}L,\bar{s}]$, $(s,y)\in\Omega$, that is, the constraint is not realized at $s=NL$. To also remove the constraint at $s=-NL$ we use the analogue of \eqref{lim-piu} for $a_-$ and Proposition \ref{convergence} that together with $\bar{h}<N-1$ imply that the translation $u^N((\cdot+L,\cdot))$ of one period of $u^N$ to the right does not realize the constraint both at $s=NL$ and $s=-NL$. The proof of Theorem \ref{teo-1} is complete.
\end{proof}

\nocite{*}
\bibliographystyle{plain}

\begin{thebibliography}{99}

\bibitem{abg}
S.~Alama, L.~Bronsard, and C.~Gui.
\newblock Stationary layered solutions in $\R^2$ for an Allen--Cahn system with multiple well potential.
\newblock {\em Calc.\ Var.} {\bf 5} No.~4 (1997), pp.~359--390.

\bibitem{a}
N.~D.~Alikakos.
\newblock  Some basic facts on the system $\Delta u-W_u(u)=0$.
\newblock {\em Proc.\ Amer.\ Math.\ Soc.} {\bf 139} No.~1 (2011), pp.~153--162.

\bibitem{a1}
N.~D.~Alikakos.
\newblock A new proof for the existence of an equivariant entire solution connecting the minima of the potential for the system $\Delta u - W_u (u) = 0$.
\newblock \emph{Comm.\ Partial Diff.\ Eqs.} {\bf 37} No.~12 (2012), pp.\ 2093--2115.

\bibitem{af2}
N.~D.~Alikakos and G.~Fusco.
\newblock On the connection problem for potentials with several global minima.
\newblock {\em Indiana\ Univ.\ Math.\ J.} {\bf 57} No.~4 (2008), pp.~1871--1906.

\bibitem{af1}
N.~D.~Alikakos and G.~Fusco.
\newblock Entire solutions to equivariant elliptic systems with variational structure.
\newblock {\em Arch.\ Rat.\ Mech.\ Analysis} {\bf 202} No.~2 (2011), pp.~567--597.

\bibitem{alikakos-fusco-smyrnelis}
N.~D.~Alikakos, G.~Fusco, and P.~Smyrnelis.
\newblock Monograph.
\newblock In preparation.

\bibitem{baldo}
S.~Baldo.
\newblock Minimal interface criterion for phase transitions in mixtures of Cahn--Hil\-liard fluids.
\newblock {\em Ann.\ Inst.\ Henri Poincar\'e, Anal.\ Non Lin\'eaire} {\bf 7} No.~2 (1990), pp.~67--90.

\bibitem{ball-crooks}
J.~M.~Ball and E.~C.~M.~Crooks.
\newblock Equivariant entire solutions to the elliptic system $\Delta u=W_u(u)$ for general $G$- invariant potentials.
\newblock  {\em Calc.\ Var.} {\bf 40} (2011), pp.~501--538.

\bibitem{bates-ren}
P.~W.~Bates and X.~Ren.
\newblock Transition layer solutions of a higher order equation in an infinite tube.
\newblock \emph{Comm.\ Partial Diff.\ Eqs.} {\bf 21} No.~1--2 (1996), pp.~195--220.

\bibitem{bn}
H.~Berestycki and L.~Nirenberg.
\newblock Travelling fronts in cylinders.
\newblock {\em Ann.\ Inst.\ Henri Poincar\'e, Anal.\ Non Lin\'eaire} {\bf 9} No.~5 (1992), pp.~497--572.

\bibitem{bronsard-gui-schatzman}
L.~Bronsard, C.~Gui, and M.~Schatzman.
\newblock A three-layered minimizer in $\R^2$ for a variational problem with a symmetric three-well potential.
\newblock {\em Comm.\ Pure.\ Appl.\ Math.} {\bf 49} No.~7 (1996), pp.~677--715.


\bibitem{evans}
L.~C.~Evans.
\newblock {\em Partial differential equations}.
\newblock Graduate Studies in Mathematics {\bf 19}, American Mathematical Society, second edition, 2010.

\bibitem{evans-gariepy}
L.~C.~Evans and R.~F.~Gariepy.
\newblock {\em Measure theory and fine properties of functions}.
\newblock CRC Press, Boca Raton, FL, 1992.

\bibitem{f}
G.~Fusco.
\newblock Equivariant entire solutions to the elliptic system $\Delta u=W_u(u)$ for general $G$- invariant potentials.
\newblock {\em Calc.\ Var.}, in press.

\bibitem{flp}
G.~Fusco, F.~Leonetti, and C.~Pignotti.
\newblock A uniform estimate for positive solutions of semilinear elliptic equations.
\newblock {\em Trans.\ Amer.\ Math.\ Soc.} {\bf 363}  No.~8 (2011), pp.~4285--4307

\bibitem{gui}
C.~Gui.
\newblock Hamiltonian identities for elliptic differential equations.
\newblock \emph{J.\ Funct.\ Anal.} {\bf 254} No.~4 (2008), pp.~904--933.

\bibitem{gui-schatzman}
C.~Gui and M.~Schatzman.
\newblock Symmetric quadruple phase transitions.
\newblock \emph{Indiana Univ.\ Math.\ J.} {\bf 57} No.~2 (2008), pp.~781--836.


\bibitem{ren}
X.~Ren.
\newblock A variational approach to multiple layers of the bistable equation in long tubes.
\newblock {\em Arch.\ Rat.\ Mech.\ Analysis} {\bf 138} No.~2 (1997), pp.~169--203.

\bibitem{ste}
P.~Sternberg.
\newblock Vector-valued local minimizers of nonconvex variational problems.
\newblock {\em Rocky Mountain J.\ Math.} {\bf 21} (1991), pp.~799--807.

 \end{thebibliography}

 \end{document}